\documentclass[10pt]{elsarticle}
\usepackage{dcolumn}
\usepackage{bm}
\usepackage{verbatim}       

\usepackage{amsthm}
\usepackage{amssymb}
\usepackage{amscd}

\usepackage{bm}
\usepackage{amstext}
\usepackage{psfrag}
\usepackage{amsmath}
\usepackage{amsfonts}
\usepackage{units}
\usepackage{mathrsfs}
\newcommand{\bd}{\begin{definition}}                
\newcommand{\ed}{\end{definition}}                  
\newcommand{\bc}{\begin{corollary}}                 
\newcommand{\ec}{\end{corollary}}                   
\newcommand{\bl}{\begin{lemma}}                     
\newcommand{\el}{\end{lemma}}                       
\newcommand{\bp}{\begin{proposition}}            
\newcommand{\ep}{\end{proposition}}                
\newcommand{\bere}{\begin{remark}}                  
\newcommand{\ere}{\end{remark}}                     

\newcommand{\bt}{\begin{theorem}}
\newcommand{\et}{\end{theorem}}

\newcommand{\be}{\begin{equation}}
\newcommand{\ee}{\end{equation}}

\newcommand{\bit}{\begin{itemize}}
\newcommand{\eit}{\end{itemize}}
\newtheorem{theorem}{Theorem}[section]
\newtheorem{corollary}[theorem]{Corollary}
\newtheorem{lemma}[theorem]{Lemma}
\newtheorem{proposition}[theorem]{Proposition}
\theoremstyle{definition}
\newtheorem{definition}[theorem]{Definition}
\theoremstyle{remark}
\newtheorem{remark}[theorem]{Remark}

\journal{Topology and its Applications}

\begin{document}

\begin{frontmatter}

\title{Quasi-pseudo-metrization of topological preordered spaces}

\author{E. Minguzzi}

\address{ Dipartimento di Matematica Applicata ``G. Sansone'',
Universit\`a degli Studi di Firenze, Via S. Marta 3,  I-50139
Firenze, Italy.} \ead{ettore.minguzzi@unifi.it}




\begin{abstract}
We establish that every second countable completely regularly
preordered space $(E,\mathscr{T},\le)$ is quasi-pseudo-metrizable,
in the sense that there is a quasi-pseudo-metric $p$ on $E$ for
which the pseudo-metric $p\vee p^{-1}$ induces $\mathscr{T}$ and the
graph of $\le$ is exactly the set $\{(x,y): p(x,y)=0\}$. In the
ordered case it is proved that these spaces can be characterized as
being order homeomorphic to subspaces of the ordered  Hilbert cube.
The connection with quasi-pseudo-metrization results obtained in
bitopology is clarified. In particular, strictly
quasi-pseudo-metrizable ordered spaces are characterized as being
order homeomorphic to order subspaces of the  ordered  Hilbert cube.
\end{abstract}


\end{frontmatter}

\section{Introduction}
A fundamental theorem by Urysohn and Tychonoff establishes that
every second countable regular space ($T_3$-space) is metrizable.
This work aims at generalizing this result for topological spaces
endowed with a preorder $\le$. In this case one would like to prove
the existence of a function $p: E\times E\to [0,+\infty)$ which
encodes both the topology and the preorder, where the latter is
obtained through the condition $x\le y$ iff $p(x,y)=0$.
Clearly, function $p$ cannot be a metric in the usual sense, in fact
we shall need the more general notion of quasi-pseudo-metric.


Topological preordered spaces appear in  various fields, for
instance in the study of dynamical systems \cite{akin93}, general
relativity \cite{minguzzi09c}, microeconomics \cite{bridges95} and
computer science \cite{gierz03}. Quasi-pseudo-metrizable preordered
spaces are among the most well behaved topological preordered
spaces, thus it is important to establish if one can just work with
quasi-pseudo-metrizable preordered spaces in the mentioned
applications. Topological preordered spaces are connected to
bitopological spaces but the latter spaces are less directly
connected with the said applications. This is so because,
generically, a flow on a manifold, a causal order on spacetime, or a
preference of an agent in microeconomics, to make a few examples,
are represented by  a preorder which does not necessarily come from
a nicely behaved bitopological space. The category of topological
preordered spaces is in this respect more interesting, and so far
much less investigated, than that of bitopological spaces.

Some definitions will help us to make our mathematical  problem more
precise. A {\em topological preordered space} is a triple
$(E,\mathscr{T},\le)$ where $(E,\mathscr{T})$ is a topological space
and $\le$ is a preorder, namely a reflexive and transitive relation.
Our terminology for topological preordered spaces will follow
Nachbin \cite{nachbin65}. Two topological preordered spaces
$E_1,E_2,$ are {\em preorder homeomorphic}  if there is an
homeomorphism $\varphi: E_1\to E_2$ such that $x\le y$ if and only
if $\varphi(x)\le \varphi(y)$. A subset $S\subset E$ of a
topological preordered space $E$ is a {\em subspace} once it is
endowed with the induced topology $\mathscr{T}_S$ and preorder
$\le_S$ defined by:  for $x,y \in S$, $x \le_S y$ if $x\le y$. The
topological preordered space $E_1$ is preorder embedded in $E_2$ if
it is preorder homeomorphic with a subspace of $E_2$.

A  preorder is an {\em order} if it is antisymmetric. With
\mbox{$i(x)=\{y: x\le y\}$} and $d(x)=\{y: y\le x\}$ we denote the
increasing and decreasing hulls. The topological preordered space is
{\em semiclosed preordered} if $i(x)$ and $d(x)$ are closed for
every $x \in E$, and it is {\em closed preordered} if the graph of
the preorder
\[
G(\le)=\{(x,y): x\le y\},
\]
is closed. A subset $S\subset E$, is called {\em increasing or
upper} if $i(S)=S$ and {\em decreasing or lower} if $d(S)=S$.  It is
called {\em monotone} if it is increasing or decreasing.
A subset $C$ is convex if it is the intersection of a decreasing and
an increasing set in which case it follows $C=d(C)\cap i(C)$. In
this work it is understood that the set inclusion is reflexive,
$S\subset S$.

A topological preordered space  is a {\em normally preordered space}
if it is semiclosed preordered and for every closed decreasing set
$A$ and closed increasing set $B$ which are disjoint, $A\cap
B=\emptyset$, it is possible to find an open decreasing set $U$ and
an open increasing set $V$ which separate them, namely $A\subset U$,
$B\subset V$, and $U\cap V=\emptyset$.

A topological preordered space  is a {\em regularly preordered
space} if it is semiclosed preordered, (a) for every closed
decreasing set $A$ and closed increasing set $B$ of the form
$B=i(x)$ which are disjoint, $A\cap B=\emptyset$, it is possible to
find an open decreasing set $U$ and an open increasing set $V$ which
separate them, namely $A\subset U$, $B\subset V$, and $U\cap
V=\emptyset$, and (b) for every closed decreasing set $A$ of the
form $A=d(x)$ and closed increasing set $B$  which are disjoint,
$A\cap B=\emptyset$, it is possible to find an open decreasing set
$U$ and an open increasing set $V$ which separate them, namely
$A\subset U$, $B\subset V$, and $U\cap V=\emptyset$.

 An {\em isotone} function is a function $f: E\to \mathbb{R}$
such that $x\le y \Rightarrow f(x)\le f(y)$.

In a normally preordered space if $A$ is closed decreasing, $B$ is
closed increasing  and $A\cap B=\emptyset$, there is some continuous
isotone function $f:E\to [0,1]$, such that  $A\subset f^{-1}(0)$ and
$B\subset f^{-1}(1)$ (this is the preorder analog of Urysohn
separation lemma, see \cite[Theor. 1]{nachbin65}). Normally
preordered spaces are closed preordered spaces, and closed
preordered spaces are semiclosed preordered spaces.

A topological preordered space is {\em convex} at $x\in E$, if for
every open  neighborhood $O\ni x$, there are an open decreasing set
$U$ and an open increasing  set $V$ such that $x\in U\cap V\subset
O$. A topological preordered space $E$ is {\em convex}  if it is
convex at every point \cite{nachbin65,kent85,kunzi04}. Notice that
according to this terminology the statement ``the topological
preordered space $E$ is convex'' differs from the statement ``the
subset $E$ is convex'' (which is always true). The terminology is
not uniform in the literature, for instance Lawson \cite{lawson91}
calls {\em strongly order convexity} what we call convexity.


 A quasi-uniformity \cite{nachbin65,fletcher82}
is a pair $(X,\mathcal{U})$ such that $\mathcal{U}$ is a filter on
$X\times X$, whose elements contain the diagonal $\Delta=\{(x,y):
x=y\}$, and such that if $V\in \mathcal{U}$ then there is $W\subset
\mathcal{U}$, such that $W\circ W\subset V$. A quasi-uniformity is a
uniformity if $V\in \mathcal{U}$ implies $V^{-1} \in \mathcal{U}$.
To any quasi-uniformity $\mathcal{U}$ corresponds a dual
quasi-uniformity $\mathcal{U}^{-1}=\{U: U^{-1}\in \mathcal{U}\}$.

From a quasi-uniformity $\mathcal{U}$ it is possible to construct a
topology $\mathscr{T}(\mathcal{U})$ in such a way that a base for
the filter of neighborhoods at $x$ is given by the sets of the form
$U(x)$ where $U(x)=\{y: (x,y)\in U\}$ with $U \in \mathcal{U}$. In
other words $O\in \mathscr{T}(\mathcal{U})$ if for every $x\in O$
there is $U\in \mathcal{U}$ such that $ U(x) \subset O$.

Given a quasi-uniformity $\mathcal{U}$ the family $\mathcal{U}^*$
given by the sets of the form $V\cap W^{-1}$, $V,W\in\mathcal{U}$,
is the coarsest uniformity containing $\mathcal{U}$. The symmetric
topology of the quasi-uniformity is $\mathscr{T}(\mathcal{U}^{*})$.
Moreover, the intersection $\bigcap \mathcal{U}$ is the graph of a
preorder on $X$ (see \cite{nachbin65}), thus given a
quasi-uniformity one naturally obtains a topological preordered
space $(X,\mathscr{T}(\mathcal{U}^{*}),\bigcap \mathcal{U})$. The
topology $\mathscr{T}(\mathcal{U}^{*})$ is Hausdorff if and only if
the preorder $\bigcap \mathcal{U}$ is an order \cite{nachbin65}.

A {\em completely regularly preordered} space (Tychonoff-preordered
space), is a topological preordered space for which  the following
two conditions hold:
\begin{itemize}
\item[(i)] $\mathscr{T}$ coincides with the initial topology generated
by the set of continuous isotone functions $g:E\to [0,1]$,
\item[(ii)] $x\le y$ if and only if for every continuous isotone
function $f:E\to [0,1]$, $f(x)\le f(y)$.
\end{itemize}
Convex normally preordered spaces are completely regularly
preordered spaces, and completely regularly preordered spaces are
convex closed preordered spaces \cite{nachbin65}. Nachbin  proves
\cite[Prop. 8]{nachbin65} that a topological preordered space
$(E,\mathscr{T},\le)$ comes from a quasi-uniformity $\mathcal{U}$,
in the sense that $\mathscr{T}=\mathscr{T}(\mathcal{U}^{*})$ and
$G(\le)=\bigcap \mathcal{U}$ if and only if it is a completely
regularly preordered space,  and proves that every Hausdforff
quasi-uniformizable space admits a closed order compactification
(the Nachbin compactification).

For the discrete preorder $G(\le)=\Delta$, the definitions of
normally preordered space, completely regularly preordered space,
regularly preordered space, closed preordered space, reduce
respectively to normal space, completely regular space, regular
space ($T_3$-space), Hausdorff space.

A {\em bitopological space} is a triple
$(E,\mathscr{P},\mathscr{Q})$ where $E$ is a set and
$\mathscr{P},\mathscr{Q}$, are two topologies on $E$. It is possible
to associate to every topological preordered space a bitopological
space by taking as $\mathscr{P}$ the topology made of all the upper
sets $\mathscr{T}^{\sharp}$, and as $\mathscr{Q}$ the topology made
of all the lower sets $\mathscr{T}^{\flat}$.  Bitopological spaces
were introduced by Kelly \cite{kelly63} and subsequently
investigated in \cite{patty67,lane67}.

A {\em quasi-pseudo-metric} \cite{kelly63,patty67}  on a set $X$ is
a function $p:X\times X \to [0,+\infty)$ such that for $x,y,z\in X$
\begin{itemize}
\item[(i)] $p(x,x)= 0$,
\item[(ii)] $p(x,z)\le p(x,y)+p(y,z)$.
\end{itemize}
The quasi-pseudo-metric is called {\em quasi-metric} \cite{wilson31}
if (i) is replaced with (i'): $p(x,y)= 0$ iff $x=y$. Other
variations exist in the literature. For instance, if (i) is replaced
by (i'') $p(x,y)=p(y,x)=0$ iff $x=y$, we get what is sometimes
referred to as {\em Albert's quasi-metric} \cite{albert41}.

The quasi-pseudo-metric is called {\em pseudo-metric} if
$p(x,y)=p(y,x)$. If a quasi-metric is such that $p(x,y)=p(y,x)$ then
it is a {\em metric} in the usual sense. Sometimes
quasi-pseudo-metrics are called {\em semi-metrics} \cite{nachbin65}
but for other authors semi-metrics are quite different objects
\cite{willard70}. If $p$ is a quasi-pseudo-metric then $q$, defined
by
\[q(x,y)=p(y,x),\] is a quasi-pseudo-metric called {\em conjugate} of
$p$. This structure, called {\em quasi-pseudo-metric space}, is
denoted $(X,p,q)$ and we might equivalently use the notation
$p^{-1}$ for $q$.

From a quasi-pseudo-metric space $(E,p,q)$ we can construct a
quasi-uniformity $\mathcal{U}$ and the associated topological
preordered space $(E, \mathscr{T}(\mathcal{U}^{*}),\bigcap
\mathcal{U})$ following Nachbin \cite{nachbin65}, or a bitopological
space $(X,\mathscr{P},\mathscr{Q})$ following Kelly \cite{kelly63}.

Nachbin defines the quasi-uniformity $\mathcal{U}$ as the filter
generated by the countable base
\begin{equation} \label{vaz}
W_n=\{(x,y)\in X\times X: p(x,y)<1/n\}.
\end{equation}
thus the graph of the preorder is $G(\le)=\bigcap
\mathcal{U}=\{(x,y): p(x,y)=0\}$ and the topology
$\mathscr{T}(\mathcal{U}^{*})$ is that of the
pseudo-metric
$p+q$. In particular this topology is Hausdorff if and only if $p+q$
is a metric i.e.\ $p(x,y)+p(y,x)=0 \Rightarrow x=y$, which is the
case if and only if the preorder $\le$ is an order. Clearly, every
topological preordered space obtained in this way is a completely
regularly preordered space as it comes from a quasi-uniformity.

\begin{remark}
Clearly, the pseudo-metrics $p\vee q$, $p+q$, $(p^2+q^2)^{1/2}$,
induce the same topology. Nevertheless, we shall preferably use
$p+q$ because in the proof of theorem \ref{mgy} we shall take
advantage of the linearity of this expression. The choice $d = p +
q$ has also been made by Kelly \cite[p. 87]{kelly63}.
\end{remark}

Given a quasi-pseudo-metric $p$ we shall denote $P(x,r)=\{y: p(x,y)<
r\}$ the $p$-ball of radius $r$ centered at $x$, and analogously
$Q(x,r)=\{y: q(x,y)<r\}$ will be the $q$-ball for the conjugate
metric $q$. If $d=p+q$, the $d$-balls will be denoted $D(x,r)=\{y:
d(x,y)< r\}$.

From a quasi-pseudo-metric space $(E,p,q)$  Kelly constructs a
bitopological space $(X,\mathscr{P},\mathscr{Q})$ as follows:
$\mathscr{P}$ is the topology having as base the sets of the form
$P(x,r)$ for arbitrary $x \in X$ and $r> 0$. Analogously,
$\mathscr{Q}$ is the topology having as base the sets of the form
$Q(x,r)$ for arbitrary $x \in X$ and $r> 0$.

\begin{remark} \label{voy}
A base of open neighborhoods at $z \in X$ is given by the sets of
the form $P(z,\epsilon)$, $\epsilon> 0$. Indeed, if $z \in \{y:
p(x,y)<r\}$, that is $p(x,z)<r$, then there is $\epsilon$ such that
$\{w:p(z,w)<\epsilon\} \subset \{y: p(x,y)<r\}$. This follows from
$p(x,w)\le p(x,z)+p(z,w)$, as choosing $\epsilon=r-p(x,z)$, we get
$p(x,w)<r$.
\end{remark}

\section{Quasi-pseudo-metrizability and preorders}

We give and motivate the following definitions.

\begin{definition}$\empty$ \\
A topological preodered space $(E,\mathscr{T},\le)$ is {\em
quasi-pseudo-metrizable} if  there is a pair of conjugate
quasi-pseudo-metrics $p,q$, said {\em admissible}, such that
$\mathscr{T}$ is the topology generated by the pseudo-metric $p+q$,
and the graph of the preorder is given by $G(\le)=\{(x,y):
p(x,y)=0\}$.

A topological preodered space $(E,\mathscr{T},\le)$ is {\em strictly
quasi-pseudo-metrizable} if it is convex semiclosed preordered and
there is a pair of conjugate quasi-pseudo-metrics $p,q$, such that
the topology associated to $p$ is the upper topology
$\mathscr{T}^\sharp$, and the topology associated to $q$ is the
lower topology $\mathscr{T}^\flat$.

A (strictly) quasi-pseudo-metrizable preoreded space is a {\em
(strictly) quasi-pseudo-metrized preordered space} if a choice of
conjugate metrics complying with the previous requirement is made.
\end{definition}

It must be noted that we call {\em strictly quasi-pseudo-metrizable
space} what, taking as reference the literature on bitopological
spaces, one would simply call {\em quasi-pseudo-metrizable space}.
The point is that in the topological preordered space version of a
topological property one has usually two or more possibilities and
the stronger can often be interpreted as the  bitopological version
of the property. For instance, Lawson \cite{lawson91} defines  the
{\em strictly completely regularly ordered spaces} which admit the
bitopological interpretation of {\em pairwise complete regularity}
\cite{kunzi90}, in contrast with Nachbin's completely regularly
ordered spaces which do not admit a bitopological interpretation.

\begin{proposition} \label{hwx}
Let $(E,\mathscr{T},\le)$ be quasi-pseudo-metrizable preordered
space and let $p,q$ be a pair of admissible conjugate quasi-pseudo
metrics. The function $p: E\times E \to \mathbb{R}$ is continuous in
the product topology $\mathscr{T}\times \mathscr{T}$ on $E$.
Moreover, it is Lipschitz with respect to $d=p+q$, in the sense that
\begin{equation} \label{nvo}
\vert p(x,y)-p(w,z)\vert\le d(x,w)+d(y,z).
\end{equation}
(This inequality holds also for $d$ replaced by $p\vee q$ or
$(p^2+q^2)^{1/2}$.) For a fixed $x\in E$, the function $q(x,\cdot)$
is isotone and the function $p(x,\cdot)=q(\cdot,x)$ is anti-isotone.
\end{proposition}

\begin{proof}
The repeated application of the triangle inequality gives
\begin{align*}
p(x,y)&\le p(x,w)+p(w,z)+p(z,y)\le d(x,w)+p(w,z)+d(y,z),\\
p(w,z)&\le p(w,x)+p(x,y)+p(y,z)\le d(x,w)+p(x,y)+d(y,z),
\end{align*}
thus $\vert p(x,y)-p(w,z)\vert\le d(x,w)+d(y,z)$. By assumption, $d$
generates $\mathscr{T}$ thus $p$ is continuous in the product
topology $\mathscr{T}\times \mathscr{T}$.

If $y\le z$ then $p(y,z)=q(z,y)=0$ and $q(x,y)\le
q(x,z)+q(z,y)=q(x,z)$ namely $q(x,\cdot)$ is isotone. If $y\le z$
then $p(y,z)=0$ and $p(x,z)\le p(x,y)+p(y,z)=p(x,y)$ namely
$p(x,\cdot)$ is anti-isotone.
\end{proof}

\begin{proposition}
Every strictly quasi-pseudo-metrizable preordered space is a
quasi-pseudo-metrizable preordered space. Every
quasi-pseudo-metrizable preordered space is a completely regularly
preordered space.
\end{proposition}

\begin{proof}
Equation (\ref{nvo}) can be obtained as in the proof of Prop
\ref{hwx} and written $\vert p(x,y)-p(w,z)\vert\le
p(x,w)+q(x,w)+p(y,z)+q(y,z)$, from which it follows that $p$ is
continuous in the product topology
$\sup(\mathscr{T}^\sharp,\mathscr{T}^\flat) \times
\sup(\mathscr{T}^\sharp,\mathscr{T}^\flat)$ (and analogously for
$q$). Thus $p+q$ is
$\sup(\mathscr{T}^\sharp,\mathscr{T}^\flat)\times
\sup(\mathscr{T}^\sharp,\mathscr{T}^\flat)$-continuous, which
implies that the topology $\mathscr{D}$ of the pseudo-metric $d=p+q$
is coarser than $\sup(\mathscr{T}^\sharp,\mathscr{T}^\flat)$.
However, since $p,q\le d$ the  $p$-balls and $q$-balls centered at a
point are $\mathscr{D}$-neighborhoods of the point, thus by remark
\ref{voy}, $\sup(\mathscr{T}^\sharp,\mathscr{T}^\flat)$ is coarser
than $\mathscr{D}$, thus
$\mathscr{D}=\sup(\mathscr{T}^\sharp,\mathscr{T}^\flat)$. Clearly,
$\sup(\mathscr{T}^\sharp,\mathscr{T}^\flat)$ is coarser than
$\mathscr{T}$, but since $E$ is convex,
$\sup(\mathscr{T}^\sharp,\mathscr{T}^\flat)=\mathscr{T}$ which
implies $\mathscr{D}=\mathscr{T}$.

Since $(E,\mathscr{T},\le)$ is semiclosed preordered, $i(x)$ is
closed thus $i(x)=cl_{\mathscr{T}^\flat} x$. It follows that $y\in
i(x)$ iff every $q$-ball centered at $y$ includes $x$ which is
equivalent to ``for all $n\ge 1$, $x\in \{ w: q(y,w)<1/n\}$,'' which
in turn is equivalent to $p(x,y)=0$. We conclude that $y\in i(x)$
iff $p(x,y)=0$.

For the last statement, every quasi-pseudo-metrizable topological
preordered space comes from a quasi-uniformity and hence is a
completely regularly preordered space.
\end{proof}

The problem of quasi-pseudo-metrization of a bitopological space was
considered already in Kelly's work \cite{kelly63} and has been
extensively studied over the years
\cite{salbany72,parrek80,romaguera83,romaguera84,raghavan88,romaguera90,andrikopoulos07,marin09}.
As we shall see in a moment, the solution to this problem can be
used to infer results on the quasi-pseudo-metrizability of a
topological preordered space. The quasi-pseudo-metrizability of a
{\em topological} space has also been investigated
\cite{stoltenberg67,sion67,norman67,kunzi83,kopperman93} but it is
less interesting for our purposes because just one topology cannot
bring information on a non trivial preorder.

For bitopological spaces Kelly \cite[Theor. 2.8]{kelly63} obtained a
generalization of Urysohn's metrization theorem which in our
topological preordered space framework reads as follows

\begin{theorem} (Kelly)
Let $(E,\mathscr{T},\le)$ be a convex regularly preordered space and
assume that both $\mathscr{T}^\sharp$ and $\mathscr{T}^\flat$ are
second countable, then $(E,\mathscr{T},\le)$ is strictly
quasi-pseudo-metrizable.
\end{theorem}

Unfortunately, this theorem is not so easily applied to topological
preordered spaces because  the second countability of $\mathscr{T}$
does not imply the second countability of the coarser topologies
$\mathscr{T}^\sharp$ and $\mathscr{T}^\flat$. This type of
difficulty is met for the various  quasi-pseudo-metrizability
results that can be found in the literature on bitopological spaces.
Nevertheless, we shall show that by weakening the thesis it is
indeed possible to prove

\begin{theorem} \label{bhs}
The following conditions are equivalent for a topological preordered
space $(E,\mathscr{T},\le)$
\begin{itemize}
\item[(a)] $(E,\mathscr{T},\le)$ is a second countable completely
regularly preordered space,
\item[(b)] $(E,\mathscr{T},\le)$ is separable and
quasi-pseudo-metrizable.
\end{itemize}


\end{theorem}


Let us recall that in a pseudo-metric space, separability, second
countability and the Lindel\"of property are equivalent \cite[Theor.
16.11]{willard70}.

Suppose that $(E,\mathscr{T},\le)$ is separable and
quasi-pseudo-metrizable. Then as $\mathscr{T}$ is induces from the
pseudo-metric $p+q$, $(E, \mathscr{T})$ is second countable, and by
Nachbin's characterization cited above (paragraph of Eq.
(\ref{vaz})), $(E,\mathscr{T},\le)$ is completely regularly
preordered. Thus we have proved (b) $\Rightarrow$ (a), and it
remains to prove (a) $\Rightarrow$ (b). To this end, we shall make
use of the following result due to Nachbin \cite[Theor.
8]{nachbin65}, which generalizes the well known metrization theorem
of Alexandroff and Urysohn.


%
%

\begin{theorem} (Nachbin) \label{nac}
A quasi-uniformizable preordered space (i.e.\ completely regularly
preordered space) comes from a quasi-pseudo-metric if and only if
the quasi-uniformity admits a countable base.
\end{theorem}

%

Given a preorder $\le$ we obtain an equivalence relation $\sim$
through ``$x\sim y$ if $x\le y$ and $y \le x$''. In the next proof
$E/\!\!\sim$ is the quotient space, $\mathscr{T}/\!\!\sim$ is the
quotient topology, and  $\lesssim$ is defined by, $[x]\lesssim [y]$
if $x\le y$ for some representatives. The quotient preorder is by
construction an order. The triple
$(E/\!\!\sim,\mathscr{T}/\!\!\sim,\lesssim)$ is a topological
ordered space and $\pi: E\to E/\!\!\sim$ is the continuous quotient
projection.

\begin{proof}[Proof of theorem \ref{bhs}, (a) $\Rightarrow$ (b)]
As a first step let us show that there is a countable family
$\mathcal{C}$ of continuous isotone functions $c_k: E \to [0,1]$,
$k\ge 1$, such that $x\le y$ if and only if $\forall k,  c_k(x)\le
c_k(y)$. Indeed, defined for every continuous isotone function $c$,
$U_c=\{(x,y): c(x)\le c(y)\}$, we have by complete preorder
regularity $G(\le)=\bigcap_{c} U_c$. Note that since $c$ is
continuous $U_c$ is closed in the product topology. But $E$ is
second countable thus $E\times E$ is second countable and hence
 any arbitrary intersection of closed sets can be
reduced to a countable intersection \cite[p. 180]{dugundji66}.
Therefore, there is  a countable family $\mathcal{C}$ of continuous
isotone functions $c_k$ such that $G(\le)=\bigcap_{c\in \mathcal{C}}
U_c$ which is the thesis.
%

Since $(E,\mathscr{T},\le)$  is completely regularly preordered it
is convex, thus by \cite[Prop. 2.3]{minguzzi11c} every open set is
saturated with respect to $\pi$, namely if $O\in \mathscr{T}$ then
$\pi^{-1}(\pi(O))=O$, which implies that $\pi$ is open (actually a
quasi-homeomorphism). Since $(E,\mathscr{T})$ is second countable
and $\pi$ is open, we have that $(E/\!\!\sim,\mathscr{T}/\!\!\sim)$
is second countable.

Since $(E,\mathscr{T},\le)$ is a completely regularly preordered
space then $(E/\!\!\sim,\mathscr{T}/\!\!\sim,\lesssim)$ is a
completely regularly ordered space (immediate from the definitions)
hence closed ordered which implies that
$(E/\!\!\sim,\mathscr{T}/\!\!\sim)$ is Hausdorff. But again, since
$(E/\!\!\sim,\mathscr{T}/\!\!\sim,\lesssim)$ is a completely
regularly ordered space, $(E/\!\!\sim,\mathscr{T}/\!\!\sim)$ is
Tychonoff. By Urysohn's theorem $(E/\!\!\sim,\mathscr{T}/\!\!\sim)$
is metrizable with a metric $\tilde{\rho}$.


Now, the strategy is to construct the quasi-uniformity as the weak
quasi-uniformity $\mathcal{W}$ of a countable family $\mathcal{P}$
of continuous isotone functions with values in $[0,1]$. Let us
recall that if $f:E\to [0,1]$ is a function then the sets $\{(x,y):
f(x)-f(y)< 1/k\}$ for every natural $k\ge 1$ form a (countable) base
for a quasi-uniformity on $E$. If $\mathcal{P}$ counts more than one
function then $\mathcal{W}$ is given by the smallest filter
containing all the single quasi-uniformities. The quasi-uniformity
$\mathcal{W}$ admits a countable base if $\mathcal{P}$ is countable,
indeed a base is given by all the finite intersections of the base
elements generating the single function quasi-uniformities.

As a first step we include the family $\mathcal{C}$ into
$\mathcal{P}$, in this way we obtain that $\bigcap
\mathcal{W}=G(\le)$ and that this equation cannot be spoiled by the
inclusion in $\mathcal{P}$ of arbitrary families of continuous
isotone functions. Therefore, we have  only to show that we can find
a countable family $\mathcal{Q}$ of continuous isotone functions on
$E$ with values in $[0,1]$,  such that the weak quasi-uniformity of
that family induces a topology as fine as $\mathscr{T}$ (since every
(anti)isotone function on $E$ passes to the quotient, with some
abuse of notation, we will denote in the same way the functions on
$E$ or on $E/\!\!\sim$). Let $\tilde{\rho}$ be the metric on
$E/\!\!\sim$ mentioned above. For every $n\ge 1$ we consider a
covering on $E/\!\!\sim$ by open $\tilde{\rho}$-balls of radius
$1/n$, then for every point $[x]\in E/\!\!\sim$ we construct a pair
of functions $f^{(n)}_{[x]}, g^{(n)}_{[x]}: E/\!\!\sim \, \to
[0,1]$, the former continuous and isotone and the latter continuous
and anti-isotone such that $f^{(n)}_{[x]}([x])=g^{(n)}_{[x]}([x])=1$
and $\textrm{min}(f^{(n)}_{[x]}, g^{(n)}_{[x]})([y])=0$ whenever
$\tilde\rho([x],[y])\ge 1/n$ (they exist by definition of completely
regularly ordered space). The open sets $V^{(n)}([x])=
(E/\!\!\sim)\backslash\{[y]: \textrm{min}(f^{(n)}_{[x]},
g^{(n)}_{[x]})([y])=0\}$ give an open covering of $E/\!\!\sim$ and
each of these sets is contained in an open ball of radius 1/n. By
the Lindel\"of property implied by second countability \cite[Theor.
16.9]{willard70} there is a countable subcovering which corresponds
to points $[x^{(n)}_i]$ and functions $f^{(n)}_{[x^{(n)}_i]},
g^{(n)}_{[x^{(n)}_i]}$. We add for each $n$ and $i$ the (lifted)
continuous isotone functions $f^{(n)}_{[x^{(n)}_i]}$ and
$1-g^{(n)}_{[x^{(n)}_i]}$ to $\mathcal{P}$ in such a way that the
weak quasi-uniformity $\mathcal{W}$ satisfies
$\mathscr{T}=\mathscr{T}(\mathcal{W}^{*})$.

Indeed, if $O\ni x$, with $O\in \mathscr{T}$  then $\pi(O)\ni [x]$
and we have already proved that  $\pi(O)\in \mathscr{T}/\!\!\sim$
and $\pi^{-1}(\pi(O))=O$. There is some $n$ such that the open
$\tilde\rho$-ball of radius $2/n$ centered at $[x]$ is contained in
$\pi(O)$. Since the sets $\{V^{(n)}([x^{(n)}_i])\}_i$ give a
covering there is some $i$ such that  $[x]\in
V^{(n)}([x^{(n)}_i])\subset \pi(O)$, where the inclusion follows
from the fact that the set $V^{(n)}([x^{(n)}_i])$ is contained in a
$\tilde{\rho}$-ball of radius $1/n$. In particular,
$f^{(n)}_{[x^{(n)}_i]}(x)>0$ and $g^{(n)}_{[x^{(n)}_i]}(x)>0$. Let
$j\ge 1$ be such that $f^{(n)}_{[x^{(n)}_i]}(x)>1/j$ and
$g^{(n)}_{[x^{(n)}_i]}(x)>1/j$ and let $U\cap V^{-1}\in
\mathcal{W}^{*}$ be given by
\begin{align*}
U&=\{(x,y): f^{(n)}_{[x^{(n)}_i]}(x)-f^{(n)}_{[x^{(n)}_i]}(y)<
1/j\},\\
V&=\{(x,y):
(1-g^{(n)}_{[x^{(n)}_i]}(x))-(1-g^{(n)}_{[x^{(n)}_i]}(y))< 1/j\},
\end{align*}
to which corresponds a neighborhood of $x$ in the topology
$\mathscr{T}(\mathcal{W}^{*})$ given by $(U\cap V^{-1})(x)=\{y:
f^{(n)}_{[x^{(n)}_i]}(x)-f^{(n)}_{[x^{(n)}_i]}(y)< 1/j \textrm{ and
} g^{(n)}_{[x^{(n)}_i]}(x)-g^{(n)}_{[x^{(n)}_i]}(y)< 1/j\}\subset
\{y: f^{(n)}_{[x^{(n)}_i]}(y)>0 \textrm{ and }
g^{(n)}_{[x^{(n)}_i]}(y)>0\}=\pi^{-1}(V^{(n)}([x^{(n)}_i]))\subset
O$. This last inclusion proves that
$\mathscr{T}=\mathscr{T}(\mathcal{W}^{*})$.

We have shown that $(E,\mathscr{T},\le)$ is quasi-uniformizable
where the quasi-uniformity admits a countable base thus
$(E,\mathscr{T},\le)$ is quasi-pseudo-metrizable.


\end{proof}

It should be noted that in (a) $\Rightarrow$ (b) we do not assume
that  $E$ is regularly preordered. This does not mean that the
assumption is stronger than expected because a completely regularly
preordered space need not be regularly preordered \cite[Example
1]{kunzi94b}. This is a crucial difference with respect to the usual
discrete-preorder version.

We do not use preorder regularity in theorem \ref{bhs} because this
condition is not necessary in order to obtain (b), namely a
separable
 quasi-pseudo-metrizable space need not be regularly
preordered. An example has been given in \cite[Example 1]{kunzi94b}.
This example shows also that there are separable
quasi-pseudo-metrizable spaces which are not strictly
quasi-pseudo-metrizable. Indeed, the latter spaces are perfectly
normally preordered because of a result due to Patty \cite[Theor.
2.3]{patty67} and hence regularly preordered.

A comparison with the discrete-preorder version is clarified by the
following result
\begin{theorem}
Every  second countable  convex regularly preordered space is a
completely regularly preordered space.
\end{theorem}

\begin{proof}
In \cite[Theor. 5.3]{minguzzi11f} it has been proved that  every
second countable regularly preordered space is (perfectly) normally
preordered. Since every convex normally preordered space is a
completely regularly preordered space the thesis follows.
\end{proof}

\begin{lemma} \label{blx}
Let $(E,\mathscr{T},\le)$ be a second countable completely regularly
preordered space, then there is a countable family $\mathcal{F}$ of
continuous isotone functions, $k\ge 1$, $f_k: E \to [0,1]$ such that
(i) $\mathscr{T}$ is the initial topology generated by
$\mathcal{F}$, and (ii) $x\le y$ if and only if for every $k\ge 1$,
$f_k(x)\le f_k(y)$.
\end{lemma}

\begin{proof}
An inspection of the proof of theorem \ref{bhs} shows that we have
already proved that there is a countable family $\mathcal{P}$ of
continuous isotone functions, $k\ge 1$, $f_k: E \to [0,1]$ such that
(i) $\mathscr{T}$ is the initial topology generated by
$\mathcal{P}$, and (ii) $x\le y$ if and only if for every $k\ge 1$,
$f_k(x)\le f_k(y)$.
\end{proof}

%
%
%
%
%
%
%
%
%

\section{The ordered Hilbert cube}

In this section we investigate the ordered Hilbert cube and its
connection with (strict) quasi-pseudo-metrization.

\begin{theorem} \label{xvc}
The property of being a quasi-pseudo-metrizable preordered space is
hereditary.
\end{theorem}

\begin{proof}
Assume $E$ is quasi-pseudo-metrizable and let $p,q$ be a pair of
conjugate quasi-pseudo-metrics. Let $S$ be a subspace then $x\le_S
y$ if and only if $p_S(x,y)=0$ where $p_S=p\vert_{S\times S}$.
Furthermore the induced topology $\mathscr{T}_S$ has a base of
neighborhoods given by the  $d$-balls intersected with $S$, $d=p+q$,
thus by the $d_S$-balls, where $d_S=p_S+q_S=d\vert_{S\times S}$.
\end{proof}

In general it is not true that every open increasing (decreasing)
set on the subspace $S$ is the intersection of an  open increasing
(resp. decreasing) set on $E$ with $S$. If this is the case $S$ is
called a {\em preorder subspace}
\cite{priestley72,mccartan68,kunzi90}. In a closed preordered space
every compact subspace $S$ is a preorder subspace \cite[Prop.
 2.6]{minguzzi11f}.

\begin{theorem} \label{bqs}
The property of being a strictly quasi-pseudo-metrizable preordered
space is hereditary with respect to preorder subspaces.
\end{theorem}

\begin{proof}
It is well known that convexity and the semiclosed preordered space
property are hereditary. For the remainder of the proof it suffices
to define $p_S=p\vert_{S\times S}$, $q_S=q\vert_{S\times S}$ where
$S$ is a preorder subspace. Indeed, if $V\subset S$ is open
increasing in $S$, there is $V'$ open increasing in $E$ such that
$V=V'\cap S$. Let $x\in V$ then there is some $\epsilon>0$ such that
$P(x,\epsilon)\subset V'$ which implies  $P_S(x,\epsilon)\subset
V'$, where $P_S(x,\epsilon)$ is the $p_S$-ball of radius $\epsilon$
centered at $x$. The proof in the decreasing case is analogous.

\end{proof}

\begin{lemma}
If $(E,\mathscr{T},\le)$ is a quasi-pseudo-metrizable preordered
space, then it admits a quasi-pseudo-metric bounded by 1. If
$(E,\mathscr{T},\le)$ is a strictly quasi-pseudo-metrizable
preordered space,  then it admits a quasi-pseudo-metric bounded by 1
(in the sense of strict quasi-pseudo-metric spaces i.e.\ it
generates $\mathscr{T}^{\sharp}$ with the conjugate that generates
$\mathscr{T}^{\flat}$).
\end{lemma}

\begin{proof}
The function $h: [0,+\infty) \to [0,1]$, $h(a)=\textrm{min}(a,1)$,
is non-decreasing and sublinear, $h(a+b)\le h(a)+h(b)$. If $p$ is a
quasi-pseudo-metric then $p_1=h(p)$ satisfies the triangle
inequality by the sublinearity of $h$ and satisfies also
$p_1(x,x)=h(p(x,x))=h(0)=0$ thus it is a quasi-pseudo-metric.
Defined $q_1(x,y)=p_1(y,x)=h(q(x,y))$, we have that $d_1=p_1+q_1$ is
a pseudo-metric which generates the same topology of $d=p+q$ (they
have the same balls with radius smaller than 1)
 and furthermore, $p_1(x,y)=0$ iff
$p(x,y)=0$.

The proof in the strict case is similar. The quasi-pseudo-metrics
$p_1$ and $q_1$ are defined in the same way from $p$ and $q$, since
$p_1$ shares with $p$ the same balls of radius less than $1$, and
since $q_1$ shares with $q$ the same balls of radius less than $1$,
the thesis follows.
\end{proof}

Let $(E_n,\mathscr{T}_n,\le_n)$, $n \in \mathbb{N}$, be topological
preordered spaces and let $(E,\mathscr{T},\le)$ be the topological
preordered space in which $(E,\mathscr{T})$ is the product  space
$E=\Pi_{n \in \mathbb{N}} E_n$ endowed with the product topology and
$\le$ is the product preorder: $x\le y$ if for all $n \in
\mathbb{N}$, $x_n \le_n y_n$. We have the following

\begin{theorem} \label{mgy}
The product topological preordered space $E=\Pi_{n } E_n$ is
quasi-pseudo-metrizable if and only if each $E_n$ is
quasi-pseudo-metrizable.
\end{theorem}

\begin{proof}
If $E$ is quasi-pseudo-metrizable $E_n$ is quasi-pseudo-metrizable
because it is preorder homeomorphic with a subset $S$ of $E$
obtained by fixing all the coordinates $x_k$ of $x$ to some value in
$E_k$ but for  $k=n$. One can then use theorem \ref{xvc}.

For the converse, let $p_n:E_n\times E_n\to [0,1]$ be a
quasi-pseudo-metric for $E_n$ bounded by $1$ and endow $E$ with the
quasi-pseudo-metric \[p(x,y)=\sum_{n=1}^{\infty} p_n(x_n,y_n)/2^n.\]
The proof that $p$ is a quasi-pseudo-metric is straightforward. Let
$q(x,y)=p(y,x)$ and analogously for $q_n$, $n\in \mathbb{N}$. The
pseudo-metric $d=p+q$ reads $d(x,y)=\sum_{n=1}^{\infty}
d_n(x_n,y_n)/2^n$ where $d_n=p_n+q_n$ is the pseudo-metric which
generates the topology $\mathscr{T}_n$. According to \cite[Theor.
14, Chap. 4]{kelley55} $d$ generates the product topology
$\mathscr{T}$. Finally, $p(x,y)=0$ if and only if for all $n\in
\mathbb{N}$, $p_n(x_n,y_n)=0$ which is equivalent to: for all $n\in
\mathbb{N}$, $x_n\le_n y_n$, that is, $x\le y$.

\end{proof}

One must be careful in trying to generalize the previous theorem to
the strict case. It is known that the countable product of
quasi-pseudo-metrizable spaces in the bitopological sense is
quasi-pseudo-metrizable in the bitopological sense
\cite{kelly63,salbany72}. This fact does not imply the existence of
a simple corresponding theorem in the strict
quasi-pseudo-metrization case for topological preordered spaces. The
reason is that the product bitopology can be different from the
bitopology induced by the product topology and product preorder.


For $I$-spaces \cite{priestley72} (compare \cite{mccartan68}),
namely for topological preordered spaces for which the increasing
and decreasing hulls of open sets are open, it is possible to prove
a useful strict case generalization.

\begin{theorem}
If the product topological preordered space $E=\Pi_{n } E_n$ is
strictly quasi-pseudo-metrizable, then each factor  $E_n$ is
strictly quasi-pseudo-metrizable, furthermore if $E$ is also an
$I$-space then so are the factors $E_n$. If each factor $E_n$ is a
strictly quasi-pseudo-metrizable $I$-space, then $E$ is a strictly
quasi-pseudo-metrizable $I$-space. Finally, in this last case the
upper topology on $E$ is the product of the upper topologies of the
factors, and analogously for the lower topology.
\end{theorem}

\begin{proof}
Each $E_n$ is  preorder homeomorphic with a subset $S$ of $E$
obtained by fixing all the coordinates $x_k$ of $x$ to some value in
$E_k$ but for  $k=n$. The subset $S$ just defined is actually a
preorder subspace because if $V\subset S$ is open increasing then
(omitting the preorder homeomorphism of $S$ with $E_n$)
$\pi_n^{-1}(V)$ is open increasing and $\pi_n^{-1}(V)\cap S=V$, and
analogously for the open decreasing sets. By theorem \ref{bqs}  $E$
is strictly quasi-pseudo-metrizable thus $S$ and hence $E_n$ is
strictly quasi-pseudo-metrizable. Furthermore, if $O$ is an open set
of $E_n$ then $\pi_n^{-1}(O)$ is an open set of $E$ and if $E$ is an
$I$-space $i(\pi_n^{-1}(O))=\pi_n^{-1}(i_{E_n}(O))$ is open, thus
$\pi_n(\pi_n^{-1}(i_{E_n}(O)))=i_{E_n}(O)$ is open because the
projection maps are open \cite[Theor. 8.6]{willard70}. The proof
that $d_{E_n}(O)$ is open is analogous. We conclude that each $E_n$
is an $I$-space.

For the converse, let us prove that $E$ is convex. Let $O$ be an
open set in the product topology and let $x\in O$. There  are open
sets $O_{i_1}\subset E_{i_1}$, $\cdots$ $O_{i_s}\subset E_{i_s}$,
$x_{i_k}\in O_{i_k}$, such that $\Pi_{n=1}^{\infty} W_n \subset O$,
where $W_n=O_{i_k}$ if $n=i_k$ for some $1\le k\le s$, or $W_n=E_n$
otherwise. Recalling that each topological preordered space $E_i$ is
convex, the sets $O_{i_k}$ can be chosen to be intersections
$O_{i_k}=U_{i_k}\cap V_{i_k}$ where $U_{i_k}$ is open decreasing and
$V_{i_k}$ is open increasing in $E_{i_k}$. Evidently defined
$U'=\Pi_{n=1}^{\infty} Y_n$ where $Y_n=U_{i_k}$ if $n=i_k$ for some
$1\le k\le s$, or $Y_n=E_n$ otherwise, and $V'=\Pi_{n=1}^{\infty}
Z_n$ where $Z_n=V_{i_k}$ if $n=i_k$ for some $1\le k\le s$, or
$Z_n=E_n$ otherwise, we have $x\in U'\cap V'\subset
\Pi_{n=1}^{\infty} W_n \subset O$ which proves that $E$ is convex
because $U'$ is open decreasing in $E$ and $V'$ is open increasing
in $E$.

Let us prove that $E$ is semiclosed preordered. Indeed, if $x\in E$,
using the definition of product order, $i(x)=\bigcap_n
E\backslash\pi_n^{-1}(E_n\backslash i_{E_n}(x_n))$ from which we
obtain that $i(x)$ is closed in $E$ because each $i_{E_n}(x_n)$ is
closed in $E_n$. Analogously, $d(x)$ is closed.

Let $p_n:E_n\times E_n\to [0,1]$ be a quasi-pseudo-metric for $E_n$
bounded by $1$ and endow $E$ with the quasi-pseudo-metric
$p(x,y)=\sum_{n=1}^{\infty} p_n(x_n,y_n)/2^n$. The proof that $p$ is
a quasi-pseudo-metric is straightforward. Let $V\subset E$ be an
open increasing set and let $x\in V$ then by definition of product
topology there are open sets $O_{i_1}\subset E_{i_1}$, $\cdots$
$O_{i_s}\subset E_{i_s}$, $x_{i_k}\in O_{i_k}$, such that defined
$G=\Pi_{n=1}^{\infty} W_n $, where $W_n=O_{i_k}$ if $n=i_k$ for some
$1\le k\le s$, or $W_n=E_n$ otherwise, we have $G\subset V$.  The
sets $i_{E_{i_k}}(O_{i_k})$ are open and increasing by the $I$-space
assumption. We define the open set on $E$, $Q=\Pi_{n=1}^{\infty}
R_n$ where $R_n=i_{E_{i_k}}(O_{i_k})$ if $n=i_k$ for some $1\le k\le
s$, or $R_n=E_n$ otherwise. Using the definition of product
preorder, $Q=i(G)$ (note that every base element on $E$ has the form
of $G$, as we proved that $Q$ is open, this formula shows, among the
other things, that $E$ is an $I$-space).

By strict quasi-pseudo-metrizability of  $E_{i_k}$ there are numbers
$r_{i_k}>0$ such that $P_{i_k}(x_{i_k},r_{i_k})\subset
i_{E_{i_k}}(O_{i_k})$ where $P_{i_k}(x_{i_k},r_{i_k})$ is a
$p_{i_k}$-ball centered at $x_{i_k}$. Let $\epsilon$ be the minimum
of $r_{i_k}/2^{i_k}$ for $k=1,\cdots, s$.

Let us prove that  $P(x,\epsilon)\subset Q\subset V$. The last
inclusion follows from the fact that $V$ is increasing and $G\subset
V$. For the former inclusion, if $y\in P(x,\epsilon)$ then
$p_{i_k}(x_{i_k},y_{i_k})/2^{i_k}<\epsilon \le r_{i_k}/2^{i_k}$ thus
$y_{i_k}\in P_{i_k}(x_{i_k},r_{i_k})\subset i_{E_{i_k}}(O_{i_k})$.
If we define $w \in E$ to be that point such that $w_{n}\in O_{n}$,
$y_{n}\in i_{E_n}(w_{n})$ for $n=i_k$, $k=1,\cdots, s$, and
$w_n=y_n$ otherwise, we have $y=i(w)$ and $w \in \Pi_{n=1}^{\infty}
W_n=G$, thus $y\in i(G)=Q$ which is the thesis.

The inclusion $P(x,\epsilon)\subset V$ proves that $p$ generates
$\mathscr{T}^\sharp$. The proof that $q$ generates
$\mathscr{T}^\flat$ is analogous.

The inclusion $Q\subset V$ proves that $\mathscr{T}^\sharp$
coincides with the product of the upper topologies on $E_n$.
Analogously, the product of the lower topologies on $E_n$ gives
$\mathscr{T}^\flat$.

\end{proof}

The canonical quasi-pseudo-metric for the real line $\mathbb{R}$
with the usual order is $m(x,y)=\textrm{max}(x-y,0)$. With this
choice $\mathbb{R}$ becomes a strict quasi-pseudo-metric $I$-space.
The interval $[0,1]$ is a preorder subspace of the real line, thus
the quasi-pseudo-metric on $\mathbb{R}$ induces on the interval
$[0,1]$ a quasi-pseudo-metric which is bounded by 1, and which makes
$[0,1]$ a strict quasi-pseudo-metrized space which is actually an
$I$-space. From the previous theorem  we get

\begin{proposition}
The Hilbert cube $H=[0,1]^{\mathbb{N}}$ once endowed with the
product topology and the product order is a strict
quasi-pseudo-metric ordered $I$-space with quasi-pseudo-metric
$p(x,y)= \sum_{n=1}^{\infty} \max(x_n-y_n,0)/2^n$.
\end{proposition}


\begin{theorem} \label{vgm}
The following conditions are equivalent for a topological ordered
space $(E,\mathscr{T},\le)$
\begin{itemize}
\item[(a)] $(E,\mathscr{T},\le)$ is a second countable completely
regularly ordered space,
\item[(b)] $(E,\mathscr{T},\le)$ is order embeddable in the
ordered Hilbert cube $H$.
\end{itemize}
\end{theorem}

\begin{proof}
(b) $\Rightarrow (a)$. Since $H$ is the countable product of
Hausdorff second countable spaces it is second countable
\cite[Theor. 16.2]{willard70}. As $E$ is homeomorphic to a subset
$S$ of $H$ it is second countable. Furthermore, the subspace $S$ is
quasi-pseudo-metrizable because this property is hereditary and
hence it is a completely regularly ordered space. As $E$ is order
homeomorphic with $S$ the thesis follows.

(a) $\Rightarrow (b)$. Let $\mathcal{F}$ be the family of continuous
isotone functions $f_k: E\to [0,1]$ given by lemma \ref{blx}. They
separate points because if it were $f_k(x)=f_k(y)$ for all $k$, then
we would infer from $f_k(x)\le f_k(y)$ for all $k$, $x\le y$, and
from $f_k(y)\le f_k(x)$ for all $k$, $y\le x$, from which it follows
$x=y$. By the embedding lemma \cite[Theor. 8.12]{willard70} the
function $f: E\to H$ whose components are the functions $f_k: E\to
[0,1]$, is an embedding. Actually, it is a preorder embedding
because $x\le y$ if and only if for all $k$, $f_k(x)\le f_k(y)$,
which is equivalent to $f(x)\le f(y)$, as $H$ is endowed with the
product order.


\end{proof}

At least in the compact case it is possible to infer that the
topological preordered space is strictly quasi-pseudo-metrizable
through the following

\begin{theorem} \label{lro}
Every compact quasi-pseudo-metrized preordered space is a strictly
quasi-pseudo-metrized preordered space (with the same
quasi-pseudo-metric).
\end{theorem}

\begin{proof}
We already know that the $p$-balls $P(x,r)=\{y: p(x,y)<r\}$ are
increasing and open because of the continuity of $p$. We have to
prove that every open increasing set is the union of $p$-balls. The
proof for the decreasing case will be analogous. Let $V$ be an open
increasing set and let $x\in V$, we have $V\supset i(x)=\{y:
p(x,y)=0\}=\cap_{i=1}^{\infty} C_i$, $C_i=\{y: p(x,y)\le 1/i\}$,
where $C_i$ are closed sets.  Thus $\emptyset=(M\backslash V)\cap
(\cap_{i=1}^{\infty} C_i)=\cap_{i=1}^{\infty}[(M\backslash V)\cap
C_i]$, but the sets $(M\backslash V)\cap C_i$ are closed, compact
and, if non-empty, they satisfy the finite intersection property
which contradicts the previous empty intersection \cite[Theor. 1,
Chap. 5]{kelley55}. Thus some of them must be empty, that is for
some $i$, $C_i\subset V$, which reads $P(x,1/i)\subset V$.
\end{proof}

\begin{proposition} \label{xsk}
Let $(E,\mathscr{T},\le)$ be a strictly quasi-pseudo-metrizable
preordered space which is  separable then on  $E$ the topologies
$\mathscr{T}$, $\mathscr{T}^\sharp$ and $\mathscr{T}^\flat$ are
 second countable.

\end{proposition}

\begin{proof}
Separability with respect to one topology implies separability with
respect to any coarser topology thus $E$ is separable with respect
to $\mathscr{T}^\sharp$ and $\mathscr{T}^\flat$. The function
$d=p+q$ is a pseudo-metric compatible with the topology
$\mathscr{T}$. By \cite[Theor. 16.11]{willard70}, separability with
respect to $\mathscr{T}$ implies second countability of
$\mathscr{T}$. Let us prove the second countability of
$\mathscr{T}^\sharp$, the proof for $\mathscr{T}^\flat$ being
similar. Let $\{c_1,c_2, \ldots\}$ be a countable set which is dense
according to $\mathscr{T}$ and define
\[
U_{nm}=\{x: p(c_n,x)< 1/m\}, \ n=1,2,\ldots,\ m=1,2,\ldots
\]
so that $\{U_{nm}:  n=1,2,\ldots,\ m=1,2,\ldots\}$ is countable. We
claim it is a base indeed let $y\in V\in \mathscr{T}^\sharp$. By
remark \ref{voy} there is some $m$ such that $P(y,1/m)\subset V$.
Consider the set $D(y,1/(2m))$. This set is open in the topology
$\mathscr{T}$ thus there is some $n$ such that $d(y,c_n)< 1/(2m)$
which implies $p(y,c_n)<1/(2m)$ and $p(c_n,y)<1/(2m)$. The set
$U_{n\,2m}$ is therefore such that $y\in U_{n\,2m}$ and $U_{n\,2m}
\subset V$ (as $p(y,x)\le p(y,c_n)+p(c_n,x)<1/m$).

\end{proof}

The next result clarifies that the difference between non-strict and
strict quasi-pseudo-metrizable spaces is that both can be identified
with subspaces of the ordered Hilbert cube but the latter types can
also be identified with {\em order} subspaces of the ordered Hilbert
cube.

\begin{theorem} \label{bhz}
The following conditions are equivalent for a topological ordered
space $(E,\mathscr{T},\le)$
\begin{itemize}
\item[(a)] $(E,\mathscr{T},\le)$ is a separable  strictly
quasi-pseudo-metrizable space,
\item[(b)] $(E,\mathscr{T},\le)$ is order embeddable as an order subspace of the
ordered Hilbert cube $H$.
\end{itemize}
\end{theorem}

\begin{proof}
(b) $\Rightarrow$ (a). If $E$ can be identified with an order
subspace of the ordered Hilbert cube $H$ then $E$ is a strictly
quasi-pseudo-metrizable since this property is hereditary with
respect to order subspaces. Furthermore, $H$ has a second countable
topology thus the topology of $E$ is second countable and hence
separable.

(a) $\Rightarrow$ (b). Let $p$ be a quasi-pseudo-metric bounded by
$1$ for the strict quasi-pseudo-metrizable space $E$. Let
$\{c_1,c_2, \ldots\}$ be a countable set which is dense according to
$\mathscr{T}$ and define $f_n(x)=1-p(c_n,x)$, $g_n(x)=p(x,c_n)$,
which are both continuous and isotone with value in $[0,1]$ (Prop.
\ref{hwx}) . The countable family of functions $f_n$ is denoted
$\mathcal{F}$ and the countable family of functions $g_n$ is denoted
$\mathcal{G}$. We define $\mathcal{R}=\mathcal{F}\cup \mathcal{G}$.
We are going to prove that (i1) the coarsest topology on $E$ which
makes all the elements of  $\mathcal{F}$ upper semi-continuous is
$\mathscr{T}^\sharp$ and (i2) the coarsest topology on $E$ which
makes all the elements of $\mathcal{G}$ lower semi-continuous
$\mathscr{T}^\flat$. From that it follows, by convexity of $E$, that
(i) the initial topology for $\mathcal{R}$ is $\mathscr{T}$. We are
also going to prove that (ii) $x\le y$ iff for every $r\in
\mathcal{R}$, $r(x)\le r(y)$.

Let $\mathcal{U}$ be the  coarsest topology on $E$ which makes all
the elements of  $\mathcal{F}$ upper semi-continuous. $\mathcal{U}$
admits as subbase the sets of the form $f_n^{-1}((a,1])$ with $a\in
[0,1]$ which are open and increasing thus all the sets of
$\mathcal{U}$ are open and increasing, that is, $\mathcal{U}\subset
\mathscr{T}^\sharp$.

For the converse, let $V$ be open increasing and let $x\in V$. There
is some $\epsilon>0$ such that $P(x,\epsilon)\subset V$, and there
is some $c_i\in D(x,\epsilon/2)$. The $p$-ball $P(c_i,\epsilon/2)$
contains $x$ because $p(c_i,x)\le d(c_i,x)=d(x,c_i)<\epsilon/2$, and
$P(c_i,\epsilon/2)\subset P(x,\epsilon)\subset V$ because if $z\in
P(c_i,\epsilon/2)$ we have $p(x,z)\le p(x,c_i)+p(c_i,z)\le
d(x,c_i)+p(c_i,z)<\epsilon/2+\epsilon/2=\epsilon$ thus $z\in
P(x,\epsilon)$. These observations imply that the function $f_i$ is
such that $x\in f_i^{-1}((1-\epsilon/2,1])\subset V$ which proves
that $\mathcal{U}$ is as fine as $\mathscr{T}^\sharp$ and hence
equal to it. Actually, it proves more, namely that the sets of the
form $f_n^{-1}((a,1])$ with $a\in [0,1]$ form a base for
$\mathcal{U}$ and hence for $\mathscr{T}^\sharp$. The proof of (i2)
is analogous.

As for (ii), if $x\le y$ we get $r(x)\le r(y)$ because all the
elements of $\mathcal{R}$ are isotone. If $x\nleq y$ then $x\in
E\backslash d(y)$ which is open increasing thus, by the just proved
result, there is some $f_j\in \mathcal{F}$ and $c\in  [0,1]$ such
that $x\in f^{-1}_j((c,1])\subset E\backslash d(y)$ which implies
$f_j(x)>c>f_j(y)$, thus setting $r=f_j$, $r(x)\nleq r(y)$.

The collection $\mathcal{R}$ separates points because if $r(x)=r(y)$
for all $r\in \mathcal{R}$, then, by the just proved result, $x\le
y$ and $y\le x$ which implies by the order assumption, $x=y$. By the
embedding theorem \cite[Theor. 8.12]{willard70}, the map $\rho: E\to
H$ whose components are $r_{2i}=f_i$, $r_{2i+1}=g_i$, that is the
functions of $\mathcal{R}$, is an embedding, and an order embedding
because of (ii).

Let us prove that $\rho(E)$ is not only a subspace but in fact an
order subspace of $H$. For simplicity we shall identify $E$ with
$\rho(E)$ thus we shall omit the order homeomorphism between the two
spaces. Let $V$ be an open increasing subset of $E$, we have to find
an open increasing subset $V'\subset H$, such that $V'\cap E=V$. For
every $x\in V$ there is some $r_{2i}\in \mathcal{F} \subset
\mathcal{R}$, $r_{2i}=f_i$, and $c\ge 0$ such that $x\in
f_i^{-1}((c,1])\subset V$. The open set $V'_x$ on $H$ given by the
product of all intervals [0,1] but for the (2i)-th term which is
$(c,1]$, is open increasing and such that $V'_x\cap E=
f_i^{-1}((c,1])\subset V$. Defined $V'=\bigcup_{x\in V} V'_x$ we
have $V'\cap E=V$, which is the thesis. The proof in the decreasing
case is analogous.

\end{proof}

\begin{remark}
In the ordered case theorem \ref{lro} follows also from theorem
\ref{bhz}. Indeed, suppose that $(E,\mathscr{T},\le)$ is a compact
quasi-pseudo-metrized ordered space. By the order assumption,
$\mathscr{T}$ is Hausdorff and $d$ is a metric. By compactness and
metrizability $(E,\mathscr{T})$ is separable, thus
$(E,\mathscr{T},\le)$ is a separable quasi-pseudo-metrizable space
which can be regarded as a subset of the ordered Hilbert cube. Since
it is compact  it is an order subspace \cite[Prop. 2.6]{minguzzi11f}
and hence a strictly quasi-pseudo-metrizable space.
\end{remark}

%
%
%
%
%
%
%
%
%
%
%
%

\section{Conclusions}
In the framework of topological preordered spaces we have proved
that the family of the second countable completely regularly
preordered spaces coincides with the family of separable
quasi-pseudo-metrizable spaces (Theor. \ref{bhs}). The theorem is
optimal as there are counterexamples if the latter family is
narrowed to the strictly quasi-pseudo-metrizable spaces or the
former family is enlarged to include the second countable regularly
preordered spaces. We have also shown that in the ordered case the
second countable completely regularly ordered spaces are,
essentially, subspaces of the ordered Hilbert cube (Theor.
\ref{vgm}). The difference with the strict case comes from the fact
that with the strict condition the subspace can be chosen to be an
order subspace (Theor. \ref{bhz}).

It remains open the problem of establishing conditions which,
starting from second countability and the assumption that $E$ is a
completely regularly preordered space could allow one to prove that
$E$  is strictly quasi-pseudo-metrizable. We have shown that a
compactness condition would be enough  (Theor. \ref{lro}) but this
assumption is quite strong for applications.

Another direction for further investigation is that of the
generalization of the Nagata-Smirnov-Bing metrization theorems to
the topological preordered space case. Unfortunately, it seems that
several arguments cannot be generalized and analogous
quasi-pseudo-metrization results could not hold or could be much
harder to prove.


\section*{Acknowledgments}
I warmly thank P. Pageault for some useful suggestions. This work
has been partially supported by ``Gruppo Nazionale per la Fisica
Matematica''  (GNFM) of ``Instituto Nazionale di Alta Matematica''
(INDAM).

\section*{Appendix: Quasi-uniformities adapted to uniformities and preorders}

A classical problem \cite{nachbin65} asks to establish, given a
uniformity $\mathcal{O}$ on a preordered space $(E,\le)$, if there
is some quasi-uniformity $\mathcal{U}$ such that
$\mathcal{U}^*=\mathcal{O}$ and $\bigcap \mathcal{U}=G(\le)$. In
this appendix we provide a result which is connected to this problem
as well as with the problem of quasi-pseudo-metrization of a
topological preordered space.

The canonical quasi-uniformity $\mathcal{R}$ on $\mathbb{R}$ is that
generated by the quasi-pseudo-metric $m(x,y)=\max(0,x-y)$. The dual
quasi-uniformity $\mathcal{R}^{-1}$ is generated by the
quasi-pseudo-metric $n(x,y)=\max(0,y-x)$ and $\mathcal{R}^*$  is
generated by the metric  $m+n=\vert x-y\vert$. Given a family of
functions on a topological  space $E$ with values in $\mathbb{R}$
one can induce both a weak uniformity or a weak quasi-uniformity on
$E$ depending as to whether one endow $\mathbb{R}$ with
$\mathcal{R}^*$ or $\mathcal{R}$.

\begin{theorem}
Let $(E,\le)$ be a preordered space, let $\mathcal{O}$ be a
uniformity on $E$,  and let $\mathcal{F}$ be a family of uniformly
continuous functions with value in $\mathbb{R}$ with the properties
\begin{itemize}
\item[(i)] $\mathscr{O}$ coincides with the weak uniformity generated
by the set of functions $\mathcal{F}$,
\item[(ii)] $x\le y$ if and only if for every
  $f\in \mathcal{F}$,  $f(x)\le f(y)$,
\end{itemize}
then the weak quasi-uniformity $\mathcal{U}$ generated by
$\mathcal{F}$ is such that  $\mathcal{U}^{*}=\mathcal{O}$, $\bigcap
\mathcal{U}=G(\le)$, and the functions in $\mathcal{F}$ become
quasi-uniformly continuous with respect to $\mathcal{U}$.

If $\mathcal{F}$ is countable then $\mathcal{U}$ admits  a countable
base, thus $(E,\mathcal{U})$ is quasi-pseudo-metrizable (see theorem
\ref{nac}). If $\mathcal{O}$ and $\mathcal{F}$ satisfy (i) and (ii),
$\mathcal{O}$ admits a countable base, and
$\mathscr{T}(\mathcal{O})$ is second countable (equivalently,
$\mathcal{O}$ is induced by a pseudo-metric which makes $E$ a
separable pseudo-metric space) then there is a subfamily
$\mathcal{F}'\subset \mathcal{F}$ which is countable and satisfies
(i) and (ii).

\end{theorem}

\begin{proof}
The weak uniformity $\mathcal{O}$ generated by $\mathcal{F}$ admits
a subbase made of subsets of $E\times E$ of the form $(f\times
f)^{-1} R$ where $f\in \mathcal{F}$ and $R\in \mathcal{R}^{*}$
(i.e.\ a base is made by the finite intersections of sets of that
form). For each $R\in \mathcal{R}^{*}$ there are $U,V\in
\mathcal{R}$ such that $U\cap V^{-1}\subset R$ (note that $U\cap
V^{-1}\in \mathcal{R}^*$ by definition of the latter family) thus
$\mathcal{O}$ admits a subbase made of subsets of $E\times E$ of the
form $(f\times f)^{-1} (U\cap V^{-1})=[(f\times f)^{-1}U]\cap
[(f\times f)^{-1}V^{-1}]=[(f\times f)^{-1}U]\cap [(f\times
f)^{-1}V]^{-1}$.

The weak quasi-uniformity $\mathcal{U}$ generated by $\mathcal{F}$
admits a subbase made of subsets of $E\times E$ of the form
$(f\times f)^{-1} U$ where $f\in \mathcal{F}$ and $U\in \mathcal{R}$
(i.e.\ a base is made by the finite intersections of sets of that
form). A subbase for the  quasi-uniformity $\mathcal{U}^{*}$ is then
given by subsets of $E\times E$ of the form $[(f\times f)^{-1}U]\cap
[(f\times f)^{-1}V]^{-1}$ for $U,V\in \mathcal{R}$. We conclude that
$\mathcal{U}^{*}=\mathcal{O}$. Finally,
\begin{align*}
\bigcap \mathcal{U}&=\bigcap_{f\in \mathcal{F}} \bigcap_{U\in
\mathcal{R}} (f\times f)^{-1} U=\bigcap_{f\in \mathcal{F}} (f\times
f)^{-1} \bigcap_{U\in \mathcal{R}} U \\
&= \bigcap_{f\in \mathcal{F}}
(f\times f)^{-1} G(\le_{\mathbb{R}})=\bigcap_{f\in \mathcal{F}}
\{(x,y): f(x)\le f(y)\}=G(\le).
\end{align*}
The functions in $\mathcal{F}$ are quasi-uniformly continuous with
respect to $\mathcal{U}$ by definition of weak quasi-uniformity.

If $\mathcal{F}$ is countable there is  a subbase for $\mathcal{U}$
given by $(f_i\times f_i)^{-1} U_{m}$, where $f_i\in \mathcal{F}$
and $U_{m}=\{(x,y)\in \mathbb{R}\times \mathbb{R}: x-y <1/m\}$.
Since the subbase is countable the base obtained from all the
possible finite intersections is also countable.

If $\mathcal{O}$ admits a countable base then it comes from a
pseudo-metric $d$ (e.g.\ \cite[Theor. 13, Chap. 6]{kelley55}). For a
topological pseudo-metrizable space second countability is
equivalent to separability \cite[Theor. 11, Chap. 4]{kelley55} thus
(i) $\mathcal{O}$ admits a countable base and $\mathscr{T}$ is
second countable, is equivalent to (ii) $\mathcal{O}$ comes from a
pseudo-metric $d$ such that $(E,d)$ is a separable pseudo-metric
space.

Suppose $\mathcal{O}$ has a countable base $O_n$ and that
$\mathscr{T}(\mathcal{O})$ is second countable, then for each $n$ we
can find some integers $k_n,m\ge 1$, and  some functions $f^{(n)}_1,
f^{(n)}_2, \cdots, f^{(n)}_{k_n}\in \mathcal{F}$, such that
$\bigcap_{i=1}^{k_n} (f_i^{(n)}\times f_i^{(n)})^{-1} R_{m}\subset
O_n$, where $R_{m}=\{(x,y)\in \mathbb{R}\times \mathbb{R}: \vert
x-y\vert<1/m\}$. Consider the family $\mathcal{F}'$  which includes
the functions $f^{(n)}_i$ so selected plus another countable family
of uniformly continuous functions which we shall define in a moment.
We have that the weak uniformity it generates is still
$\mathcal{O}$.

Since $\mathscr{T}(\mathcal{O})$ is second countable, the product
topology $\mathscr{T}\times \mathscr{T}$ on $E\times E$ is second
countable. Since the functions belonging to $\mathcal{F}$ are
continuous and $G(\le_{\mathbb{R}})$ is closed, each set $(f\times
f)^{-1}G(\le_{\mathbb{R}})$ for $f\in \mathcal{F}$, is closed in the
product topology of $E\times E$. By second countability of $E\times
E$, the intersection $G(\le)=\bigcap_{f\in \mathcal{F}} (f\times
f)^{-1}G(\le_{\mathbb{R}})$ can be reduced to the intersection of a
countable number of terms and we include the corresponding elements
of $\mathcal{F}$ into $\mathcal{F}'$. The family $\mathcal{F}'$ is
then countable and satisfies (i) and (ii).

\end{proof}


\begin{thebibliography}{10}
\expandafter\ifx\csname url\endcsname\relax
  \def\url#1{\texttt{#1}}\fi
\expandafter\ifx\csname urlprefix\endcsname\relax\def\urlprefix{URL
}\fi \expandafter\ifx\csname href\endcsname\relax
  \def\href#1#2{#2} \def\path#1{#1}\fi

\bibitem{akin93}
E.~Akin, The general topology of dynamical systems, Amer. {M}ath.
{S}oc.,
  Providence, 1993.

\bibitem{minguzzi09c}
E.~Minguzzi, Time functions as utilities, {C}ommun. {M}ath. {P}hys.
298 (2010)
  855--868.

\bibitem{bridges95}
D.~S. Bridges, G.~B. Mehta, Representations of preference orderings,
Vol. 442
  of Lectures Notes in Economics and Mathematical Systems, {Springer-Verlag},
  Berlin, 1995.

\bibitem{gierz03}
G.~Gierz, K.~H. Hofmann, K.~Keimel, J.~D. Lawson, M.~W. Mislove,
D.~S. Scott,
  Continuous lattices and domains, Cambridge {U}niversity {P}ress, 2003.

\bibitem{nachbin65}
L.~Nachbin, Topology and order, D. {V}an {N}ostrand {C}ompany,
{I}nc.,
  Princeton, 1965.

\bibitem{kent85}
D.~C. Kent, On the {W}allman order compactification, Pacific {J}.
{M}ath. 118
  (1985) 159--163.

\bibitem{kunzi04}
H.-P.~A. K{\"u}nzi, T.~A. Richmond, Completely regularly ordered
spaces versus
  $t_2$-ordered spaces which are completely regular, Topology {A}ppl. 135 (2004) 185--196.

\bibitem{lawson91}
J.~D. Lawson, Order and strongly sober compactifications, Vol.
Topology and
  Category Theory in Computer Science, Clarendon Press, Oxford, 1991, pp.
  171–--206.

\bibitem{fletcher82}
P.~Fletcher, W.~Lindgren, Quasi-uniform spaces, Vol.~77 of Lect.
{N}otes in
  {P}ure and {A}ppl. {M}ath., Marcel {D}ekker, {I}nc., New York, 1982.

\bibitem{kelly63}
J.~C. Kelly, Bitopological spaces, Proc. London {M}ath. {S}oc. 13
(1963)
  71--89.

\bibitem{patty67}
C.~W. Patty, Bitopological spaces, Duke {M}ath. {J}. 34 (1967)
387--392.

\bibitem{lane67}
E.~P. Lane, Bitopological spaces and quasi-uniform spaces, Proc.
{L}ondon
  {M}ath. {S}oc. 17 (1967) 241--256.

\bibitem{wilson31}
W.~A. Wilson, On quasi-metric spaces, Amer. {J}. {M}ath. 53
  (1931) 675–--684.

\bibitem{albert41}
G.~E. Albert, A note on quasi-metric spaces, Bull. Amer. Math. Soc.
47 (1941)
  479--482.

\bibitem{willard70}
S.~Willard, General topology, {Addison-Wesley} {P}ublishing
{C}ompany, Reading,
  1970.

\bibitem{kunzi90}
H.-P.~A. K{\"u}nzi, Completely regular ordered spaces, Order 7
(1990) 283--293.

\bibitem{salbany72}
S.~Salbany, Quasi-metrization of bitopological spaces, Arch. {M}ath.
23 (1972)
  299--306.

\bibitem{parrek80}
C.~Parrek, Bitopological spaces and quasi-metric spaces, The
{J}ournal of the
  {U}niversity of {K}uwait, {S}cience 6 (1980) 1--7.

\bibitem{romaguera83}
S.~Romaguera, Two characterizations of quasi-pseudometrizable
bitopological
  spaces, J. {A}ust. {M}ath. {S}oc. {S}eries {A} 35
  (1983) 327--333.

\bibitem{romaguera84}
S.~Romaguera, On bitopological quasi-pseudometrization, J. {A}ust.
{M}ath. {S}oc. {S}eries {A} 36 (1984) 126--129.

\bibitem{raghavan88}
T.~G. Raghavan, I.~L. Reilly, Characterizations of quasi-metrizable
  bitopological spaces, J. {A}ust.
{M}ath. {S}oc.,
  {S}eries {A}, 44 (1988) 271--274.

\bibitem{romaguera90}
S.~Romaguera, S.~Salbany, Quasi-metrization and hereditary normality
of compact
  bitopological spaces, Comment. {M}ath. {U}niv. {C}arolinae 31 (1990)
  113--122.

\bibitem{andrikopoulos07}
A.~Andrikopoulos, The quasimetrization problem in the
(bi)topological spaces,
  Int. J. Math. Math. Sci., ID 76904, 19 pages.

\bibitem{marin09}
J.~Mar{\'i}n, Weak bases and quasi-pseudo-metrization of bispaces,
Topology {A}ppl. 156 (2009) 3070--–3076.

\bibitem{stoltenberg67}
R.~Stoltenberg, Some properties of quasi-uniform spaces, Proc.
London Math.
  Soc. 17 (1967) 342--354.

\bibitem{sion67}
M.~Sion, G.~Zelmer, On quasi-metrizability, Canad. J. Math. 19
  (1967) 1243–--1249.

\bibitem{norman67}
L.~J. Norman, A sufficient condition for quasi-metrizability of a
topological
  space, Portugaliae Mathematica 26 (1967) 207--211.

\bibitem{kunzi83}
H.-P.~A. K{\"u}nzi, On strongly quasi-metrizable spaces, Archiv der
Mathematik
  41 (1983) 57--63.

\bibitem{kopperman93}
R.~Kopperman, Which topologies are quasi-metrizable?, Topology
{A}ppl. 52 (1993) 99–--107.

\bibitem{kunzi94b}
H.-P.~A. K{\"u}nzi, S.~Watson, A metrizable completely regular
ordered space,
  Comment. {M}ath. {U}niv. {C}arolinae 35 (1994) 773--778.

\bibitem{minguzzi11f}
E.~Minguzzi, Normally preordered spaces and utilities, {O}rder,
published
  online 05 Sept. 2011, DOI: 10.1007/s11083-011-9230-4. arXiv:1106.4457v2
  (2011).

\bibitem{dugundji66}
J.~Dugundji, Topology, Allyn and {B}acon {I}nc., Boston, 1966.

\bibitem{minguzzi11c}
E.~Minguzzi, Topological conditions for the representation of
preorders by
  continuous utilities, {A}ppl. {G}en. {T}opol. 13 (2012) 81--89.

\bibitem{priestley72}
H.~A. Priestley, Ordered topological spaces and the representation
of
  distributive lattices, Proc. {L}ondon {M}ath. {S}oc. 24 (1972) 507--530.

\bibitem{mccartan68}
S.~D. McCartan, Separation axioms for topological ordered spaces,
Proc. {C}amb.
  {P}hil. {S}oc. 64 (1968) 965--973.

\bibitem{kelley55}
J.~L. Kelley, General Topology, {Springer-Verlag}, New York, 1955.

\end{thebibliography}
\end{document}